\documentclass[12pt]{amsart}
\usepackage{latexsym}
\usepackage{amssymb, amsmath,mathrsfs}
\usepackage{geometry}
\usepackage[OT2,T1]{fontenc}

\usepackage{setspace}

\usepackage[pagebackref,hypertexnames=false, colorlinks, citecolor=red, linkcolor=red]{hyperref}

\newtheorem{theorem}{Theorem}[section]

\newtheorem{conjecture}[theorem]{Conjecture}
\newtheorem{remark}[theorem]{Remark}

\DeclareMathOperator{\Tr}{Tr}
\DeclareSymbolFont{cyrletters}{OT2}{wncyr}{m}{n}
\DeclareMathSymbol{\Sha}{\mathalpha}{cyrletters}{"58}

\begin{document}

\title{Bounds for the Hilbert Transform with Matrix $A_2$ Weights}
\date{\today}

\author[K. Bickel]{Kelly Bickel$^{\dagger}$}
\address{Kelly Bickel, Department of Mathematics\\
Bucknell University\\
701 Moore Ave\
Lewisburg, PA 17837}
\email{kelly.bickel@bucknell.edu}
\thanks{$\dagger$ Research supported in part by National Science Foundation grants
DMS  \# 0955432 and DMS \#1448846.}

\author[S. Petermichl]{Stefanie Petermichl$^{\star}$}
\address{Stefanie Petermichl, Universit\'e Paul Sabatier\\ Institut de Math\'ematiques de Toulouse\\
118 route de Narbonne\\ F-31062 Toulouse Cedex 9, France }
\email{stefanie.petermichl@math.univ-toulouse.fr}
\thanks{$\star$ Research supported in part by ANR-12-BS01-0013-02. The author is a member of IUF}

\author[B. D. Wick]{Brett D. Wick$^\ddagger$}
\address{Brett D. Wick, Department of Mathematics\\
Washington University in St. Louis\\
One Brookings Drive\\
 St. Louis, MO 63130-4899 }
\email{wick@math.wustl.edu}
\urladdr{http://www.math.wustl.edu/~wick/}
\thanks{$\ddagger$ Research supported in part by National Science Foundation
DMS grant \# 0955432 and \#1500509.}


\maketitle

\begin{abstract}
Let $W$ denote a matrix $A_2$ weight.  In this paper, we implement a scalar argument using the square function to deduce related bounds for vector-valued functions in $L^2(W).$  These results are then used to study the boundedness of the Hilbert transform and Haar multipliers on $L^2(W)$.  Our proof shortens the original argument by Treil and Volberg and improves the dependence on the $A_2$ characteristic. In particular, we prove that:
\begin{eqnarray*}
\| T\|_{L^2(W) \rightarrow L^2(W)} & \lesssim & [W]_{A_2}^{\frac{3}{2}} \textnormal{log}\, [W]_{A_2},
\end{eqnarray*}
where $T$ is either the Hilbert transform or a Haar multiplier.
\end{abstract}

\bibliographystyle{plain}

\section{Introduction}
\subsection{Scalar Setting} In this paper, we study the behavior of the Hilbert transform
\[ Hf(x) \equiv p.v. \int_{\mathbb{R}} \frac{f(y)}{x-y} dy \]
on matrix-weighted $L^2$ spaces. To set the scene, recall that in the scalar setting, the
Hunt-Muckenhoupt-Wheeden theorem says that for $1<p<\infty$, the Hilbert transform $H$ is bounded on the weighted space $L^p(w)$ if and and only if 
$w$ is in the $A_p$ Muckenhoupt class, namely, iff
\begin{equation} \label{eqn:ap} \left[ w \right]_{A_p} \equiv \sup_I \left \langle w \right \rangle_I \big \langle w^{-\tfrac{p'}{p}} \big \rangle_I^{\frac{p}{p'}} < \infty, \end{equation}
where the supremum is taken over all interals $I$, $\left \langle w \right \rangle_I$ denotes the average $ \frac{1}{|I|} \int w(x) \ dx$, and $\frac{1}{p}+\frac{1}{p'}=1.$ More generally, a Calder\'on-Zygmund operator $T$ is bounded on $L^p(w)$ as long as $w \in A_p$, for $1 < p < \infty.$ A subtle, related question that became important to the harmonic analysis community over  the past decade is:
\begin{center} What is the the dependence of $\|T \|_{L^p(w) \rightarrow L^p(w)}$ on $[w]_{A_p}?$ \end{center}
For $p=2$, the conjectured dependence was linear, and the problem was termed the $A_2$ Conjecture. This was resolved by Wittwer for the Haar multipliers \cite{wit00}, by Petermichl-Volberg for the Beurling transform \cite{pv02}, and by  Petermichl for the Hilbert transform \cite{pet07}. The case of general dyadic shifts was handled first by Petermichl implicitly in \cite{p08} and then by Lacey-Petermichl-Reguera \cite{lpr10}, using different arguments. Finally, in \cite{th12}, Hyt\"onen proved the  conjecture for general Calder\'on-Zygmund operators.  Using the sharp form of Rubio de Francia's extrapolation theorem due to Dragi\u{c}evi\'c-Grafakos-Pereyra-Petermichl \cite{dgpp05}, one can use the linear $L^2(w)$ bound to immediately obtain the sharp result
\[ \| T\|_{L^p(w) \rightarrow L^p(w)} \lesssim [w]_{A_p}^{ \max \{1, \frac{1}{p-1} \}} \qquad 1< p < \infty,\]
where the implied constant depends only on $T$, not the weight $w$. 

\subsection{Vector Setting} We are interested in generalizations of this theory to vector-valued functions.  Write $L^2\equiv L^2(\mathbb{R}, \mathbb{C}^d)$, namely those vector-valued functions satisfying
\[
\left\Vert f\right\Vert_{L^2}^2\equiv\int_{\mathbb{R}} \Vert f(x)\Vert^2_{\mathbb{C}^d}\,dx<\infty.
\]
We say a $d\times d$ matrix-valued function $W$ is a \emph{matrix weight} if $W$ has locally integrable entries and $W(x)$ is positive semidefinite for a.e.~$x$. Then one can define
$L^2(W)\equiv L^2(\mathbb{R},W, \mathbb{C}^d)$  to be the set of vector-valued functions satisfying
\begin{equation} \label{eqn:L2W}
\left\Vert f\right\Vert_{L^2(W)}^2\equiv\int_{\mathbb{R}} \Vert W^{\frac{1}{2}}(x)f(x)\Vert^2_{\mathbb{C}^d}\,dx=\int_{\mathbb{R}} \left\langle W(x)f(x), f(x)\right\rangle_{\mathbb{C}^d}\,dx<\infty.\end{equation}
As defined by Treil-Volberg \cite{vt97}, we say a weight $W$ satisfies the matrix $A_2$ Muckenhoupt condition if 
\begin{equation} \label{eqn:wa2}
\big [ W \big ] _{A_2} \equiv \sup_{I} \left \| \langle W \rangle_I^{\frac{1}{2}}
 \langle W^{-1} \rangle_I^{\frac{1}{2}} \right \|^2 < \infty,
\end{equation}
where $\left\Vert\cdot\right\Vert$ denotes the norm of the matrix acting on $\mathbb{C}^d$. One can also define $L^p(W)$ and the $A_p$ Muckenhoupt weights. However, for $p\ne 2$, the $A_p$ classes do not have as simple a definition as in \eqref{eqn:ap}. Arguably the simplest characterization appears in \cite{s03}, where Roudenko showed that $W \in A_p$ if and only if  
\[ \sup_I \frac{1}{|I|} \int_I \left( \frac{1}{|I|}\int_I \left \| W(x)^{\frac{1}{p}} W(t)^{-\frac{1}{p}} \right \|^{p'} dt \right)^{\frac{p}{p'}} dx < \infty.\]
For additional details and characterizations of $A_p$ weights, we refer the readers to \cite{gol03, lt07, nr15, s03, vol97}. 

Treil-Volberg chose to characterize $A_2$ weights as ones satisfying \eqref{eqn:wa2} because in \cite{vt97}, they proved:~the Hilbert transform is bounded on $L^2(W)$ if and only if $W$ satisfies \eqref{eqn:wa2}. They gave an alternate proof in \cite{vt97b}.  In \cite{nt97, vol97}, Nazarov-Treil and Volberg separately generalized this result to $A_p$ weights. They both also showed that a classical Calder\'on-Zygmund operator  is bounded on $L^p(W)$ if $W$ is in $A_p$. Here, ``classical'' means that the operator is defined by applying a scalar Calder\'on-Zgymund operator $T$ to each component of a vector-valued function and further, $T$ satisfies $T1=T^*1 =0$. In \cite{cg01, gol03}, Christ-Goldberg and Goldberg studied a class of weighted, vector analogues of the Hardy-Littlewood maximal function  and used them to establish the boundedness of a class of singular integral operators on $L^p(W),$ for $W \in A_p.$  

A host of related interesting problems have also been examined in the matrix setting.  For instance, in \cite{bw15b, k1, k2}, Bickel-Wick and Kerr established $T(1)$ theorems characterizing the boundedness of operators, including the Hilbert transform, between matrix weighted spaces. Meanwhile, in \cite{IKP}, Isralowitz-Kwon-Pott studied the boundedness of commutators of the form $[T,B]$ on $L^p(W),$ where $T$ is a Riesz transform and $B$ a locally integrable matrix function.  In \cite{gol02, k97, ntv97,nptv02}, Nazarov-Treil-Volberg, Katz, Goldberg, and Nazarov-Pisier-Treil-Volberg studied the dependence of the unweighted Carleson Embedding Theorem in the matrix setting on dimension $d$ and concluded that its sharp dependence is $\log d$. Similarly, in \cite{nptv02, pet00}  Nazarov-Pisier-Treil-Volberg and Petermichl studied vector-valued Hankel operators, again concluding that the operator norm's sharp dependence on dimension is $\log d.$ 

In the operator weighted setting, less is known. Gillespie-Pott-Treil-Volberg studied the Haar multipliers and Hilbert transform on weighted spaces in \cite{gptv01, gptv04}. They showed  $W$  satisfying \eqref{eqn:wa2} no longer implies that the Hilbert transform or Haar multipliers are bounded on $L^2(W).$  In \cite{pp03}, Petermichl-Pott proved a form of  Burkholder's Theorem, connecting the boundedness of the Haar multipliers with that of the Hilbert transform on operator weighted $L^2$ spaces. In \cite{kp97,p07}, Pott and Katz-Pereyra both provided interesting sufficient conditions for the Hilbert transform to be bounded on operator weighted $L^2,$ but to the best of the authors'  knowledge, necessary and sufficient  conditions have proved elusive. 

\subsection{Matrix $A_2$ Conjecture} In this paper, we are interested in the natural sharpness question. For $W$  an $A_2$ matrix weight and $T$ a Calder\'on-Zygmund operator,
\begin{center} How does $\| T \|_{L^2(W) \rightarrow L^2(W)}$ depend on $[W]_{A_2}$? \end{center}
In analogy with the scalar setting, we conjecture that the dependence is linear. However, very little is actually known about the answer, and the sharp bounds for even ``simple'' operators like the Hilbert transform and Haar multipliers are unknown. Currently, the best known results concern sparse operators. In  \cite{bw15, IKP}, Bickel-Wick and Isralowitz-Kwon-Pott separately established that if $\mathscr{S}$ is a sparse operator, then
\[ \| \mathscr{S} \|_{L^2(W) \rightarrow L2(W)} \lesssim [W]_{A_2}^{\frac{3}{2}}. \]
Similarly, in \cite{IKP}, Isralowitz-Kwon-Pott studied Christ-Goldberg's maximal function and showed that on $L^2(W)$, its norm depends linearly on $[W]_{A_2}.$

In this vector setting, progress is slow because many tools in the scalar case do not exist or generalize poorly to the matrix setting. For example, Petermichl used both a bilinear embedding theorem and Bellman function testing conditions to show that the scalar Hilbert transform's norm depends linearly on $[w]_{A_2}$ \cite{pet07}. In the matrix weighted setting, there is no known sharp bilinear embedding theorem and Bellman function arguments, while possible, are much more difficult to execute. Indeed, many arguments fail because simple scalar facts like $0< w<v$ implies $w^2 < v^2$ do not hold for matrices.

In this paper, we show that with care, some important scalar arguments  can be modified to work in the matrix setting. Specifically, we consider an elegant proof of Petermichl-Pott from \cite{petpot02}, establishing
the boundedness of the Hilbert transform on $L^2(w)$ with dependence $[w]_{A_2}^{\frac{3}{2}}.$ By modifying this argument appropriately, we prove that 
\begin{equation} \label{eqn:htbound}
\| H  \|_{L^2(W) \rightarrow L^2(W)} \lesssim [W]_{A_2}^{\frac{3}{2}} \log  [W]_{A_2} 
\end{equation}
and obtain similar results for Haar multipliers. Although these constants do not appear to be sharp, they are better than anything that has previously appeared in the literature. While it seems unlikely that the additional  $ \log  [W]_{A_2}$ is  required, removing it will certainly require nontrivial new ideas. We also mention that our results do not extend immediately to $L^p(W)$, as there are complications related to extrapolation in the matrix setting.

\subsection{Outline of Paper} In Section \ref{sec:defn}, we introduce the basic notation and tools used in the proofs of the main results. These tools include sets of disbalanced Haar functions adapted to matrix weights and a weighted matrix  embedding theorem. The remainder of the paper discusses the generalization of Petermichl-Pott's result to the matrix setting as well as current conjectures and potential modifications.

In Section \ref{sec:square}, we prove preliminary bounds involving a generalized square function, which are interesting in their own right. The main results, given in Theorems \ref{thm:SquareEst1} and \ref{thm:SquareEst2}, are the following upper and lower estimates:
\[\begin{aligned}
\sum_{I \in \mathcal{D}} \left \langle \langle W \rangle_I \widehat{f}(I),
\widehat{f}(I) \right \rangle_{\mathbb{C}^d} & \lesssim [ W  ]_{A_2}^2 \textnormal{log}\, [W]_{A_2} \ \| f \|^2_{L^2(W)};  \\
\| f \|^2_{L^2(W)} & \lesssim  [W]_{A_2} \textnormal{log}\, [W]_{A_2}  \
\sum_{I \in \mathcal{D}} \left \langle \langle W \rangle_I \widehat{f}(I),
\widehat{f}(I) \right \rangle_{\mathbb{C}^d}, 
\end{aligned}
\]
which hold for all $f \in L^2(W)$. Other proofs of the upper square function bound appear in \cite{I15, nt97, vol97}. Our paper improves the stated dependence of $[W]_{A_2}^4$ and $[W]_{A_2}^3$ in \cite{I15}. Similarly, although the \cite{nt97, vol97} proofs will give constants depending on $[W]_{A_2}$, we did not track this dependence explicitly, and it seems unlikely that any resulting constants would be near optimal. It is also worth pointing out that the analogous estimates appearing in \cite{petpot02} for the scalar case do not include a $\log [w]_{A_2}$ term. Rather, this comes from the matrix embedding theorem. Another technicality is that to use the scalar arguments, we needed to reduce to the situation of bounded matrix weights; this reduction is handled in Remark \ref{rem:bounded} and  involves truncations at the level of eigenvalues.

In Section \ref{sec:ht}, we prove the previously-discussed bound for the Hilbert transform \eqref{eqn:htbound}, which appears in Theorem \ref{thm:ht}. This follows from an estimate on first order Haar shifts, since the Hilbert transform can be represented as an average of such Haar shifts. The main argument uses our previous square function bounds and the fact that the square function norm  is unaffected by these simple Haar shifts. 

%

In Section \ref{sec:hm}, we apply similar arguments to Haar multipliers. Namely, let $\sigma=\{\sigma_I\}_{I\in\mathcal{D}}$ be a sequence of matrices indexed by the dyadic intervals and define the Haar multiplier $T_{\sigma}$ 
by
$$
T_{\sigma}f \equiv \sum_{I\in\mathcal{D}} \sigma_I \widehat{f}(I) h_I,
$$
where the Haar coefficients $\widehat{f}(I)$ and Haar functions $\{h_I \}$ are defined precisely in Section \ref{sec:defn}. Then for all $f \in L^2(W)$, we show in Theorem \ref{thm:HaarMult} that
\[ \| T_{\sigma}  \|_{L^2(W) \rightarrow L^2(W)} \lesssim  \left\Vert \sigma\right\Vert_{\infty}[W]_{A_2}^{\frac{3}{2}} \log  [W]_{A_2}, \]
where $\Vert \sigma\Vert_{\infty}=\sup_{I\in\mathcal{D}} \big\Vert \big\langle W\big\rangle_I^{\frac{1}{2}}\sigma_I \big\langle W\big\rangle_I^{-\frac{1}{2}}\big\Vert$. In \cite{IKP}, Isralowitz-Kwon-Pott established this  boundedness result for $p=2$ and a similar one for all $1<p<\infty.$ In their paper's most recent version, they also mention that, using our square function bounds in Theorems \ref{thm:SquareEst1} and \ref{thm:SquareEst2}, their proof will give the same dependence on $[W]_{A_2}$ in the $p=2$ case. A recent result by Pott-Stoica \cite{ps15}, which uses methods from \cite{t12},  reduces the study of Calder\'on-Zygmund operators to the study of Haar multipliers. Pairing this with our estimate gives
\[ \|T  \|_{L^2(W) \rightarrow L^2(W)} \lesssim  [W]_{A_2}^{\frac{3}{2}} \log  [W]_{A_2}, \]
for all Calder\'on-Zygmund operators $T$.

In Section \ref{sec:open}, we discuss related open questions and conjectures. Specifically, one would hope to remove the $\log [W]_{A_2}$ from our norm bounds. One original barrier was the lack of a sharp weighted matrix Carleson Embedding Theorem, like the one used in \cite{petpot02}. Very recently, this result was proved by Culiuc-Treil in \cite{ct15}. In Section \ref{sec:open}, we show how to modify our earlier arguments to potentially use this new Carleson Embedding Theorem to improve our bounds on the square function, Hilbert transform, and Haar multipliers. However, finishing the proof requires a testing condition, which so far has remained elusive.  Finally, one can ask if there are similar bounds for operators on $L^p(W)$, with the weight $W$ in $A_p$. This is an interesting but hard problem related to sharp extrapolation. We end our paper with a discussion of the potential complications arising in the matrix setting.

\section{Basic Facts and Notation} \label{sec:defn}

Throughout this paper, $A \lesssim B$ indicates that $A \le C(d) B$, for some constant $C(d)$
that may depend on the dimension $d$.

\subsection{Dyadic Grids}
Let $\mathcal{D}$ denote the standard dyadic grid. For $\alpha \in \mathbb{R}$ and 
$r>0$, let $\mathcal{D}^{\alpha,r}$ denote
the dyadic grid $\{\alpha+ rI: I \in\mathcal{D} \}$ and let $\{h_I\}_{I \in \mathcal{D}^{\alpha, r}}$
denote the Haar functions adapted to $\mathcal{D}^{\alpha,r}$  and normalized in $L^2.$
We will use these shifted dyadic grids in Section \ref{sec:ht}, when studying the Hilbert transform. However, in much of what follows, we omit the superscripts $\alpha,r$  because the arguments hold
for all such dyadic grids. 

To be precise, for $I \in \mathcal{D}$, let $I_+$ denote its right half and $I_-$ its left half. Then $h_I$ is defined by
\[ 
h_I \equiv  |I|^{-\frac{1}{2}} \left( \textbf{1}_{I_+} - \textbf{1}_{I_-} \right) \qquad \forall I \in \mathcal{D},\]
where $\textbf{1}_{E}$ is the characteristic function of the set $E$.
Similarly, define $h_I^1 \equiv \textbf{1}_I\frac{1}{|I|}$ for any $I \in \mathcal{D}.$ One should notice that the non-cancellative Haar functions have a different normalization. Now, let $f \in L^2$. To define $\widehat{f}(I)$, let $\nu_1, \dots, \nu_d$ be an orthonormal basis in $\mathbb{C}^d.$
Then, 
\[ 
\widehat{f}(I) \equiv \int_I f(x)h_I(x) \ dx = \sum_{j=1}^d \langle f, h_I \nu_j \rangle_{L^2} \, \nu_j.
\]
Note that this decomposition works for \emph{any} orthonormal basis. In the later proofs, we will use an orthonormal basis that depends on $I$.

\subsection{Disbalanced Haar functions}
If $W$ is a matrix weight, then its entries are locally-integrable and we can define
\[ W(I) \equiv \int_I W(x) \ dx \text{ and } \left \langle W \right \rangle_I \equiv \frac{1}{|I|} \int_I W(x) \ dx.\] 
Similarly, we have: 
\[ 
\widehat{W}(I) = \int_{\mathbb{R}} W(x)h_I(x) \ dx =  \frac{1}{2} |I|^{\frac{1}{2}} \left( \left \langle W \right \rangle_{I_+} -  \left \langle W \right \rangle_{I_-} \right).
\]
Then $L^2(W)$ is defined by \eqref{eqn:L2W} and if $f,g \in L^2(W)$, then 
\[ \langle f,g \rangle_{L^2(W)} = \int_{\mathbb{R}} \left \langle W(x) f(x), g(x) \right \rangle_{\mathbb{C}^d} dx.\]
In the main proof, we will use \emph{disbalanced Haar functions
adapted to $W$}. Treil and Volberg introduce these in the matrix setting in \cite{vt97}. To define them,
fix $I \in \mathcal{D}$ and  let $e^1_I, \dots, e^d_I$ be a set of orthonormal 
eigenvectors of $\langle W \rangle_I.$ Define
\[ 
w_I^k \equiv \left \| \langle W \rangle_I^{\frac{1}{2}} e^k_I \right \|^{-1}_{\mathbb{C}^d} 
= \left  \| \langle W \rangle_I^{-\frac{1}{2}} e^k_I \right \|_{\mathbb{C}^d}.  
\]
Then, the vector-valued functions $\{w_I^k h_I e^k_I\}_{I \in \mathcal{D}, 1 \le k \le d}$ are normalized in $L^2(W).$  Define the disbalanced Haar functions
\[ g^{W,k}_I\equiv w_I^k h_I e^k_I + h_I^1 \tilde{e}_I^k, \]
where the vector $\tilde{e}_I^k = A(W,I) e^k_I$ and 
\[ 
A(W,I) = \frac{1}{2} | I |^{\frac{1}{2}} \langle W \rangle_I^{-1} \left( \langle
W \rangle_{I_-} - \langle W \rangle_{I_+} \right) \langle W \rangle_I^{-\frac{1}{2}}.  
\]
Simple calculations, which appear in \cite{vt97}, show that
\begin{equation} \label{eqn:orthog} \left \langle g^{W,k}_I, g^{W,j}_J \right \rangle_{L^2(W)} = 0  \qquad \forall \ J \ne I, \ 1 \le j,k \le d, \end{equation}
and the functions satisfy $\| g^{W,k}_I \|_{L^2(W)} \le 5$. It is worth pointing out that for $I=J$, then the inner product
\[ \left \langle g^{W,k}_I, g^{W,j}_I \right \rangle_{L^2(W)} \]
need not be zero. Using simple computations, we can write standard Haar functions using these disbalanced ones as follows
\begin{equation} \label{eqn:disHaar}
h_I e^k_I = \left(w_I^k \right)^{-1} g^{W,k}_I - \left(w_I^k \right)^{-1}  A(W,I) h^1_I e^k_I \end{equation}
for all  $I \in \mathcal{D}$ and  $k=1, \dots, d.$

\subsection{Carleson Embedding Theorem}

To prove our main results, we initially proceed as in Petermichl-Pott's proof of the scalar case in \cite{petpot02}. Some arguments
generalize easily, but to finish the proof, we require a matrix weighted embedding theorem. 
Specifically, we use the following result of  Treil-Volberg, which appears as Theorem 6.1 in \cite{vt97}:

\begin{theorem}[Treil and Volberg, \cite{vt97}] \label{thm:tv}
 Let $W$ be a  $d \times d$  matrix weight in $A_2.$ Then for all $f\in L^2$, 
 \[
 \sum_{I \in \mathcal{D}} |I| \left \| \left \langle W \right \rangle_I^{-\frac{1}{2}} \left(  
 \left \langle W \right \rangle_{I_-} - \left \langle W \right \rangle_{I_+} \right ) \left \langle W \right \rangle_I^{-\frac{1}{2}} \right \|^2 
  \left \| \left \langle W \right \rangle_I^{-\frac{1}{2}} \left \langle W^{\frac{1}{2}} f \right \rangle_I \right \|^2 
  \lesssim
 [W]_{A_2} \textnormal{log}\, [W]_{A_2} \| f \|_{L^2}^2.
 \]
\end{theorem}
The constant $ [W]_{A_2} \textnormal{log}\, [W]_{A_2} $ is not specified in Treil-Volberg's statement of the theorem. However, a 
careful reading of the proofs of their Lemma 3.1, Lemma 3.6, Theorem 4.1, and Theorem 6.1 reveal the above constant. 

\section{Square Function Estimate} \label{sec:square}

Recall the dyadic Littlewood-Paley square function, typically defined by
\begin{equation} \label{eqn:square1} Sf(x)^2 \equiv \sum_{I \in \mathcal{D}} \left | \widehat{f}(I)\right|^2 h^1_I(x),\end{equation}
for $f$ in $L^2(\mathbb{R}, \mathbb{C}),$ which coincides with the usual definition summing square norms of martingale differences in the dyadic filtration. Here is an alternate formulation.  Let $\{ -1,1\}^{\mathcal{D}}$ denote the set of all sequences indexed by the dyadic intervals whose terms only take the values $\pm 1.$
Let $\sigma\equiv \{\sigma_I \}_{I\in \mathcal{D}}$ be an element of  $\{ -1,1\}^{\mathcal{D}}$ and let $T_{\sigma}$ denote the associated Haar multiplier
\[ T_{\sigma} f = \sum_{I \in \mathcal{D}} \sigma_I \widehat{f}(I) h_I.\]
For any fixed $x$ and collection $\mathcal{D}_x$ of dyadic intervals containing $x$, consider the collection of sequences $\{-1,1\}^{\mathcal{D}_x}$ interpreted as the probability space of random sequences of independent realisations of a random variable taking values in $\pm 1$ with equal probability. So we associate the natural probability measure that assigns measure $2^{-k}$ to each cylinder of length $k$ (Bernoulli measure).
Then, as mentioned in  \cite{petpot02}, the square function can be equivalently defined as
\begin{equation} \label{eqn:square2} 
Sf(x)^2  \equiv \mathbb{E}  \left( | T_{\sigma} f(x)|^2 \right).  \end{equation}
This is equivalent to the previous definition \eqref{eqn:square1} because
\[
 \mathbb{E} \left( | T_{\sigma} f(x)|^2 \right) = \sum_{I, J \in \mathcal{D}} \mathbb{E}\left( \sigma_I \sigma_J \right) \widehat{f}(I) \overline{ \widehat{f}(J)} h_I(x) h_J(x) =
\sum_{I \in \mathcal{D}} | \widehat{f}(I)|^2 h^1_I(x).
  \]
This follows because each $\sigma_I$ takes values $\pm 1$ with equal probability. Hence, $\mathbb{E}(\sigma_I \sigma_J) = 0$ if $I \ne J$ and $\mathbb{E}(\sigma_I \sigma_J) = 1$ if $I=J.$

\subsection{Generalized Square Function}  The classical vector analogue of  \eqref{eqn:square1} is
\[ Sf(x)^2 \equiv \sum_{I \in \mathcal{D}} \left \| \widehat{f}(I)\right\|^2_{\mathbb{C}^d} h^1_I(x).
\]
Here $Sf$ is naturally scalar-valued, as it is equal to the square function summing square norms of martingale differences of vector valued martingales. However, this definition is not useful in the weighted setting because  it does not make sense to study $S$ as an operator from $L^2(W)$ to $L^2(W).$  Instead, to incorporate weights, we define a different square function $S_W$ for each weight $W$. Pulling the weight inside an operator like this is a standard trick when studying boundedness; for instance, instead of studying $T: L^2(w) \rightarrow L^2(w)$, it is often more convenient to study $M_w T: L^2(w) \rightarrow L^2(w^{-1})$, where $M_w$ is multiplication by $w$. This trick has proven essential in the vector-valued theory. For example, Christ-Goldberg's maximal functions from \cite{cg01, gol03}  incorporate the matrix weights into the operator as follows
\begin{equation} \label{eqn:max} M_W^p f(x) \equiv \sup_{I:x \in I}  \frac{1}{|I|}\int_I   \left \| W^{\frac{1}{p}}(x) W^{-\frac{1}{p}}(y) f(y) \right \|_{\mathbb{C}^d} dy, \qquad 1 < p < \infty, \end{equation}
where the superscript $p$ in $M_W^p$ just indicates the dependence of the operator on $p$. These maximal operators map into scalar-valued spaces of functions and give important information about $A_p$ weights. For our square function, we do something similar. We still let
\[
T_{\sigma}f=  \sum_{I \in \mathcal{D}}  \sigma_I \widehat{f}(I) h_I,
\]
where  $\sigma$ is an arbitrary sequence in $\{1,-1\}^{\mathcal{D}}$. Then if we define
\[
S_W:L^2(\mathbb{R},\mathbb{C}^d)\to L^2(\mathbb{R},\mathbb{R}) \quad \text{ by } \quad
S_Wf(x)^2 \equiv \mathbb{E} \left( \left \| W(x)^{\frac{1}{2}} T_{\sigma}f(x) \right \|^2_{\mathbb{C}^d} \right),
\]
we have
\[\begin{aligned}
\| S_W f \|_{L^2(\mathbb{R},\mathbb{R})}^2
& =\int_{\mathbb{R}}  \mathbb{E}\sum_{I,J \in \mathcal{D}}  \sigma_I \sigma_Jh_I(x) h_J(x)  \left \langle W(x)\widehat{f}(I)  ,  \widehat{f}(J) \right \rangle_{\mathbb{C}^d}\, dx \\
& =\int_{\mathbb{R}}  \sum_{I, J \in \mathcal{D}}  \mathbb{E} (\sigma_I \sigma_J)h_I(x) h_J(x)  \left \langle W(x)\widehat{f}(I)  ,  \widehat{f}(J) \right \rangle_{\mathbb{C}^d}\, dx \\
&=\sum_{I \in \mathcal{D}} \left \langle \langle W \rangle_I \widehat{f}(I),
\widehat{f}(I) \right \rangle_{\mathbb{C}^d}.
\end{aligned} 
\] 
Volberg introduced the study of these sums in  \cite{vol97}. Notice that in the scalar situation, we can similarly define $S_w$ by
\[ S_w f(x)^2 \equiv \mathbb{E} \left(  |w(x)^{\frac{1}{2}} T_{\sigma}f(x) |^2 \right).\]
It is immediate that
\[ \| S_w f \|_{L^2(\mathbb{R}, \mathbb{R})}^2 = \sum_{I \in \mathcal{D}} \langle w \rangle_I |\widehat{f}(I)|^2 = \| Sf \|^2_{L^2(w)}. \]
So, in the scalar situation, these  
square functions $S_w$ are bounded from $L^2(w)$ to $L^2(\mathbb{R}, \mathbb{R})$ precisely when the standard square function $S$ is bounded from $L^2(w)$ to $L^2(w),$ and the operators have the same norm. Thus, studying $S_W$ from $L^2(W)$ to $L^2(\mathbb{R}, \mathbb{R})$ gives a natural generalization of the standard square function questions. 

\subsection{Square Function Bounds} 
In the scalar setting, the square function $S$ is bounded on $L^2(w)$ if and only if $w$ is an $A_2$ weight and the dependence on $[w]_{A_2}$ is linear.  For matrix $A_2$ weights and the new square functions $S_W$,
we obtain the following similar bound, which differs from the scalar bound by a logarithm:

\begin{theorem} \label{thm:SquareEst1} Let $W$ be a  $d \times d$ 
matrix weight in $A_2.$ Then 
\[ \| S_W f \|^2_{L^2(\mathbb{R}, \mathbb{R})}  \lesssim [ W  ]_{A_2}^2 \textnormal{log}\, [W]_{A_2} \ \| f \|^2_{L^2(W)} \quad \forall f \in L^2(W). 
\]
\end{theorem}
To establish Theorem \ref{thm:SquareEst1}, we follow the arguments in \cite{petpot02}, which first prove a lower bound on the square function. Our matrix analogue is Theorem \ref{thm:SquareEst2} and 
the proof uses both arguments from \cite{petpot02}
and  Theorem \ref{thm:tv}. As with Theorem \ref{thm:SquareEst1}, this lower bound differs from the scalar bound by a factor of $\log [W]_{A_2}.$

\begin{theorem} \label{thm:SquareEst2}  Let $W$ be a  $d \times d$  matrix weight in $A_2.$ Then  
\[ 
\| f \|^2_{L^2(W)} \lesssim  [W]_{A_2} \log [W]_{A_2}  \
 \| S_W f \|^2_{L^2(\mathbb{R}, \mathbb{R})} \quad \forall f \in L^2(W). 
\]
\end{theorem}

\begin{proof} As in \cite{petpot02}, we can assume without loss
of generality that $W$ and $W^{-1}$ are bounded. For more details, see Remark \ref{rem:bounded}. 
Then $L^2(W)$ and $L^2$ are equal as sets.
For ease of notation, define the constant 
\[ C_W \equiv  [W]_{A_2} \textnormal{log}\, [W]_{A_2} . \]
Let $e_1, \dots, e_d$ be the standard orthonormal 
basis in $\mathbb{C}^d$.  Define the discrete multiplication
operator $D_W: L^2 \rightarrow L^2$ by
\[ 
D_W:  h_I e_k \mapsto \langle W \rangle_I h_I e_k \quad \forall I \in
 \mathcal{D}, k=1, \dots, d.
\]
Observe that
\[ 
\langle D_W f, f \rangle_{L_2} = \sum_{I \in \mathcal{D}} \left
\langle \langle W \rangle_I \widehat{f}(I), \widehat{f}(I) \right \rangle_{\mathbb{C}^d}.
\] 
Let $M_W$ denote multiplication by $W$.  Then, we can rewrite the desired inequality as
\begin{equation} \label{eqn:est1}
\langle M_W f ,f \rangle_{L^2} \lesssim  C_W
\langle D_W f ,f \rangle_{L^2}, \qquad \forall f \in L^2. 
\end{equation}
As in \cite{petpot02}, we convert this to an inverse inequality.  
Since $W$ and $W^{-1}$ are bounded, $D_W$ and $M_W$ 
are bounded and invertible with
$M_W^{-1} = M_{W^{-1}}$ and $D_W^{-1}$ defined by
\[
D^{-1}_W:  h_I e_k \mapsto \langle W \rangle_I^{-1} h_I e_k \quad \forall I \in
 \mathcal{D}, k=1, \dots, d.
\]
$M_W$ and $D_W$ and their inverses have well-defined square roots, with $M_W^{\frac{1}{2}} = M_{W^{\frac{1}{2}}}$ and $D_W^{\frac{1}{2}}$ sending each $h_I e_k$ to $\langle W \rangle_I^{\frac{1}{2}} h_I e_k$. Similarly, the spectral theorem implies that  the positive, invertible, self-adjoint operator  $D_W^{-\frac{1}{2}} M_W D_W^{-\frac{1}{2}}$ has a positive, invertible square root. Then, 
(\ref{eqn:est1}) is immediately equivalent to
\[ \langle D_W^{-\frac{1}{2}} M_W   D_W^{-\frac{1}{2}}f  ,f \rangle_{L^2} \lesssim  C_W
\langle  f ,f \rangle_{L^2}, \qquad \forall f \in L^2,
\]
 which one can show is equivalent to
\begin{equation} \label{eqn:est2}
\langle D^{-1}_W f ,f \rangle_{L^2} \lesssim  C_W
\langle M^{-1}_W f ,f \rangle_{L^2}, \qquad \forall f \in L^2.
\end{equation}
So to prove Theorem \ref{thm:SquareEst2},
we need to establish:
\[
\sum_{I \in \mathcal{D}} \left \langle \langle W \rangle^{-1}_I 
\widehat{f}(I), \widehat{f}(I) \right \rangle_{\mathbb{C}^d}
\lesssim C_W \| f \|^2_{L^2(W^{-1})} \quad \forall f \in L^2.
\]
We will rewrite the sum using Haar functions adapted to $W$. First, for $I \in \mathcal{D}$, 
let $e^1_I, \dots, e^d_I$ be a set of orthonormal 
eigenvectors of $\langle W \rangle_I.$ Recall that
\[ 
w_I^k \equiv \left \| \langle W \rangle_I^{\frac{1}{2}} e^k_I \right \|^{-1}_{\mathbb{C}^d} 
= \left  \| \langle W \rangle_I^{-\frac{1}{2}} e^k_I \right \|_{\mathbb{C}^d},
\]
so $w_I^k$ is the reciprocal of the square root of the eigenvalue corresponding to eigenvector $e_I^k.$
Using these definitions, expand the sum as follows:
\[
\begin{aligned}
\sum_{I \in \mathcal{D}} \left \langle \langle W \rangle^{-1}_I 
\widehat{f}(I), \widehat{f}(I) \right \rangle_{\mathbb{C}^d}
&= \sum_{I \in \mathcal{D}} \sum_{j,k =1}^d \left \langle \langle W \rangle^{-1}_I 
\langle f, h_I e^k_I \rangle_{L^2} e^k_I,  \langle f, h_I e^j_I \rangle_{L^2} e^j_I \right \rangle_{\mathbb{C}^d}\\
& =  \sum_{I \in \mathcal{D}} \sum_{j,k =1}^d \langle f, h_I e^k_I \rangle_{L^2} \overline{\langle f, h_I e^j_I \rangle_{L^2}}  \left \langle \langle W \rangle^{-1}_I 
e^k_I,   e^j_I \right \rangle_{\mathbb{C}^d} \\
& =  \sum_{I \in \mathcal{D}} \sum_{k =1}^d \left | \langle f, h_I e^k_I \rangle_{L^2}  \right |^2  \left \langle \langle W \rangle^{-1}_I 
e^k_I,   e^k_I \right \rangle_{\mathbb{C}^d} \\
& =  \sum_{I \in \mathcal{D}} \sum_{k =1}^d \left(w^k_I\right)^2 \left | \langle f, h_I e^k_I \rangle_{L^2}  \right |^2. 
\end{aligned}
\]
Now, we can expand the $h_I e^k_I$ using the disbalanced Haar functions adapted to $W$ as in (\ref{eqn:disHaar}). This 
transforms our sum as follows:
\[
\begin{aligned}
 \sum_{I \in \mathcal{D}} \sum_{k =1}^d \left(w^k_I \right)^2 \left | \langle f, h_I e^k_I \rangle_{L^2}  \right |^2 & =
\sum_{I \in \mathcal{D}} \sum_{k =1}^d \left(w^k_I \right)^2 \left | \langle f, \left(w_I^k \right)^{-1} g^{W,k}_I - \left(w_I^k \right)^{-1}  A(W,I) h^1_I e^k_I \rangle_{L^2}  \right |^2 \\
& \le \sum_{I \in \mathcal{D}} \sum_{k =1}^d \left | \langle f, g^{W,k}_I \rangle_{L^2}  \right |^2 \\
& \ \ + 2 \sum_{I \in \mathcal{D}} \sum_{k =1}^d  \left | \langle f, g^{W,k}_I \rangle_{L^2}  \langle f,  A(W,I) h^1_I e^k_I \rangle_{L^2} \right |\\
& \ \ + \sum_{I \in \mathcal{D}} \sum_{k =1}^d \left | \langle f,A(W,I) h^1_I e^k_I \rangle_{L^2}  \right |^2 \\
& = S_1 + S_2 + S_3.
\end{aligned} 
\]
It is clear that 
\[ 
S_1 = \sum_{I \in \mathcal{D}} \sum_{k =1}^d \left | \langle f, g^{W,k}_I 
\rangle_{L^2}  \right |^2 = \sum_{I \in \mathcal{D}} \sum_{k =1}^d \left | \langle W^{-1} f, g^{W,k}_I 
\rangle_{L^2(W)}  \right |^2  \lesssim \| f \|^2_{L^2(W^{-1})},
\]
since the $\{g^{W,k}_I\}$ satisfy \eqref{eqn:orthog} and are uniformly bounded in $L^2(W).$ Since $S_2 \lesssim S_1^{\frac{1}{2}} S_3^{\frac{1}{2}}$,
the main term to understand is $S_3.$ It can be written as 
\begin{eqnarray}
S_3   &= &  \sum_{I \in \mathcal{D}} \sum_{k =1}^d \left|  \left \langle f,  \frac{1}{2} | I |^{\frac{1}{2}} \langle W \rangle_I^{-1} \left( \langle
W \rangle_{I_-} - \langle W \rangle_{I_+} \right) \langle W \rangle_I^{-\frac{1}{2}}   h^1_I e^k_I \right \rangle_{L^2} \right|^2  \label{eqn:S31}  \\
& = & \sum_{I \in \mathcal{D}} \sum_{k =1}^d  \left | \left \langle f,  \langle W \rangle_I^{-1} 
\widehat{W}(I) \langle W \rangle_I^{-\frac{1}{2}}   h^1_I e^k_I \right \rangle_{L^2}  \right|^2 \nonumber   \\
& =  & \sum_{I \in \mathcal{D}} \sum_{k =1}^d  \left| \left \langle \langle f \rangle_I,  \langle W \rangle_I^{-1} 
\widehat{W}(I) \langle W \rangle_I^{-\frac{1}{2}}  e^k_I \right \rangle_{\mathbb{C}^d} \right |^2 \nonumber \\
& =  & \sum_{I \in \mathcal{D}} \sum_{k =1}^d  \left |\left \langle \langle  W \rangle_I^{-\frac{1}{2}} \langle f \rangle_I,  \langle W \rangle_I^{-\frac{1}{2}} 
\widehat{W}(I) \langle W \rangle_I^{-\frac{1}{2}}  e^k_I \right \rangle_{\mathbb{C}^d} \right |^2.  \label{eqn:S3}
\end{eqnarray}
Now, we can bound $S_3$ as follows:
\[
\begin{aligned}
S_3   
&\le  \sum_{I \in \mathcal{D}} \sum_{k =1}^d  \left \| \langle  W \rangle_I^{-\frac{1}{2}} \langle f 
\rangle_I \right \|^2_{\mathbb{C}^d} \left \|  \langle W \rangle_I^{-\frac{1}{2}} 
\widehat{W}(I) \langle W \rangle_I^{-\frac{1}{2}}  e^k_ I \right \|^2_{\mathbb{C}^d}\\
&\lesssim   \sum_{I \in \mathcal{D}}  \left \| \langle  W \rangle_I^{-\frac{1}{2}} \langle f \rangle_I \right \|^2_{\mathbb{C}^d} \left \|  \langle W \rangle_I^{-\frac{1}{2}} 
\widehat{W}(I) \langle W \rangle_I^{-\frac{1}{2}}  \right \|^2\\
& \lesssim [W]_{A_2} \textnormal{log}\,  [W]_{A_2}  \| f \|^2_{L^2(W^{-1})}, 
\end{aligned}
\]
where we used Theorem $\ref{thm:tv}$ applied to $g = W^{-\frac{1}{2}} f.$ This also implies a similar bound for $S_2$, 
and combining our estimates for
$S_1,S_2,S_3$ completes the proof of Theorem \ref{thm:SquareEst2}. \end{proof}

Using Theorem \ref{thm:SquareEst2}, we can easily prove Theorem \ref{thm:SquareEst1}:

\begin{proof} Again, assume without loss of generality that  $W$ and $W^{-1}$ are bounded and define the constant $B_W$ by
\[ B_W = [ W  ]_{A_2}^2 \textnormal{log}\, [W]_{A_2}= [W]_{A_2}C_W. \]
Using our previous notation, Theorem \ref{thm:SquareEst1} is equivalent to the inequality
\[ \langle D_W f , f \rangle_{L^2} \lesssim B_W \langle M_W f,f \rangle_{L^2}, \quad \forall f \in L^2. \]
We require the following operator inequality
\[ D_W \le [W]_{A_2} \left( D_{W^{-1}} \right)^{-1}. \]
The $A_2$ condition implies that for every $I \in \mathcal{D}$ and vector $e_I \in \mathbb{C}^d$,
\[
\left \langle \langle W \rangle_I^{\frac{1}{2}} \langle W^{-1} \rangle_I^{\frac{1}{2}} e_I, 
 \langle W \rangle_I^{\frac{1}{2}} \langle W^{-1} \rangle_I^{\frac{1}{2}}e_I \right \rangle_{\mathbb{C}^d}
\le [W]_{A_2} \| e_I \|^2_{\mathbb{C}^d}. 
\]
Fixing $g \in L^2$ and setting $e_I = \langle W^{-1} 
\rangle_I^{-\frac{1}{2}} \widehat{g}(I)$, we can conclude
\[
\left \langle \langle W \rangle_I^{\frac{1}{2}}  \widehat{g}(I), 
 \langle W \rangle_I^{\frac{1}{2}} \widehat{g}(I)  \right \rangle_{\mathbb{C}^d}
\le [W]_{A_2} \left \langle \langle W^{-1} 
\rangle_I^{-\frac{1}{2}} \widehat{g}(I),\langle W^{-1} 
\rangle_I^{-\frac{1}{2}} \widehat{g}(I) \right \rangle_{\mathbb{C}^d}. 
\]
Then
\[
\begin{aligned}
\langle D_W g ,g \rangle_{L^2} &= \sum_{I \in \mathcal{D}} \left \langle \langle W \rangle_I^{\frac{1}{2}}  \widehat{g}(I), 
 \langle W \rangle_I^{\frac{1}{2}} \widehat{g}(I)  \right \rangle_{\mathbb{C}^d} \\
& \le  [W]_{A_2} \sum_{I \in \mathcal{D}} \left \langle \langle W^{-1} 
\rangle_I^{-\frac{1}{2}} \widehat{g}(I),\langle W^{-1} 
\rangle_I^{-\frac{1}{2}} \widehat{g}(I) \right \rangle_{\mathbb{C}^d} \\
&=  [W]_{A_2} \langle  \left( D_{W^{-1}} \right)^{-1} g ,g \rangle_{L^2}.
\end{aligned}
\]
Combining that estimate with $(\ref{eqn:est2})$ from Theorem \ref{thm:SquareEst2} applied to $W^{-1}$ gives:
\[ \langle D_W g ,g \rangle_{L^2} \le  [W]_{A_2} \langle  \left( D_{W^{-1}} \right)^{-1} g ,g \rangle_{L^2}
\lesssim [W]_{A_2} C_W \langle M^{-1}_{W^{-1}} g ,g  \rangle_{L^2} = B_W \| g \|_{L^2(W)} \]
for all $g \in L^2,$ which completes the proof. \end{proof}

\begin{remark}[Reducing to Bounded Weights] \label{rem:bounded} \normalfont
The proofs of Theorems  \ref{thm:SquareEst1} and \ref{thm:SquareEst2} only handle weights $W$ with both $W$ and $W^{-1}$ bounded.
To reduce to this case, fix $W \in A_2$ and write 
\[  W(x) =\sum_{j=1}^d \lambda_j(x) P_{E_j(x)} \quad \text{ for } x \in \mathbb{R}, \]
where the $\lambda_j(x)$ are eigenvalues of $W(x)$, the $E_j(x)$ are the associated orthogonal eigenspaces, and the
$P_{E_j(x)}$ are the orthogonal projections onto the $E_j(x).$  Define
\[ \begin{aligned}
E^n_1(x) & \equiv \text{Eigenspaces of $W(x)$ corresponding to eigenvalues } \lambda_j(x) \le \tfrac{1}{n}; \\
E^n_2(x) & \equiv  \text{Eigenspaces of $W(x)$ corresponding to eigenvalues } \tfrac{1}{n} < \lambda_j(x) < n;  \\
E_3^n(x) & \equiv  \text{Eigenspaces of $W(x)$ corresponding to eigenvalues } \lambda_j(x) \ge n. 
\end{aligned}
\]
Using these spaces, truncate $W(x)$ as follows:
\[ W_n(x) = \tfrac{1}{n} P_{E^n_1(x)} + P_{E_2^n(x)} W(x) P_{E^n_2(x)} + n P_{E^n_3(x)}. \]
It is immediate that
\[\left(W_n(x)\right)^{-1} = n P_{E^n_1(x)} + P_{E_2^n(x)} W^{-1}(x) P_{E^n_2(x)} + \tfrac{1}{n} P_{E^n_3(x)}.\]
It is easy to see that $\frac{1}{n} I_{d\times d} \le W_n, W_n^{-1} \le n I_{d \times d}.$ Each $W_n$ is also an $A_2$ weight with
\begin{equation} \label{eqn:A2n}  [W_n]_{A_2} \equiv \sup_{I }  \left \| \langle W_n \rangle_I^{\frac{1}{2}}
 \langle W_n^{-1} \rangle_I^{\frac{1}{2}} \right \|^2 \lesssim [W]_{A_2} ,\end{equation}
where the constant depends on the dimension $d.$ For the scalar case, in \cite{rvv10}, Reznikov-Vasyunin-Volberg show that the constant in \eqref{eqn:A2n} is one. So, it would be interesting to determine the best constant in \eqref{eqn:A2n}.

Our proof of \eqref{eqn:A2n} relies on the following two facts about positive self-adjoint matrices:
\[
\begin{aligned}
& \text{Fact 1:  If $A_1, A_2 \ge 0,$ then $\left \| A_1^{\frac{1}{2}} A_2^{\frac{1}{2}} \right \|^2 \approx \Tr(A_1 A_2) .$} \\
& \text{Fact 2: If $A_1,A_2, B_1,B_2 \ge 0$ and each $A_j \le B_j$,  then $\Tr( A_1 A_2) \le \Tr(B_1 B_2)$.}
\end{aligned}
\]
Here, the implied constants again depend on $d$.
Fact $1$ allows us to equate $ \left \| \langle W_n \rangle_I^{\frac{1}{2}}
 \langle W_n^{-1} \rangle_I^{\frac{1}{2}} \right \|^2  \approx \Tr(\left \langle W_n \right \rangle_I \left \langle W^{-1}_n \right \rangle_I)$.  
 Then, using Fact $2$ and the matrix inequalities
\[ 
\begin{aligned}
 & \left \langle W_n   \right \rangle_I \le  n I_{d \times d}; \\
 & \left \langle P_{E_2^n(x)} W(x) P_{E^n_2(x)} + n P_{E^n_3(x)}  \right \rangle_I \le \left \langle W \right \rangle_I,
 \end{aligned}
 \]
 for $W_n$ and similar ones for $W^{-1}_n$, one can deduce that
 \[ \Tr(\left \langle W_n \right \rangle_I \left \langle W^{-1}_n \right \rangle_I) \le 2 \Tr(I_{d\times d}) + \Tr(\left \langle W \right \rangle_I \left \langle W^{-1}\right \rangle_I ). \]
 Applying Fact $1$ and using $1 \le [W]_{A_2}$ immediately  gives \eqref{eqn:A2n}. Then, as $W_n$ and $W_n^{-1}$ are bounded, the arguments in the proof of Theorem \ref{thm:SquareEst2} imply that
 \begin{equation} \label{eqn:bdd3.3}
\| f \|^2_{L^2(W_n)} \lesssim  [W_n]_{A_2} \textnormal{log}\, [W_n]_{A_2}  \
\sum_{I \in \mathcal{D}} \left \langle \langle W_n \rangle_I \widehat{f}(I),
\widehat{f}(I) \right \rangle_{\mathbb{C}^d} \quad \forall f \in L^2(W_n). 
\end{equation}
 Using basic convergence theorems, we will see that both
 \begin{equation}  \label{eqn:lim1} \lim_{n \rightarrow \infty} \| f \|^2_{L^2(W_n)} = \| f \|^2_{L^2(W)} \end{equation}
 and
\begin{equation}  \label{eqn:lim2} \lim_{n \rightarrow \infty}  \sum_{I \in \mathcal{D}} \left \langle \langle W_n \rangle_I \widehat{f}(I),
\widehat{f}(I) \right \rangle_{\mathbb{C}^d} = \sum_{I \in \mathcal{D}} \left \langle \langle W \rangle_I \widehat{f}(I),
\widehat{f}(I) \right \rangle_{\mathbb{C}^d}, \end{equation}
for $f \in L^2 \cap L^2(W).$ In fact, to obtain the first inequality, observe that 
\[ \left \langle W_n(x)f(x), f(x) \right \rangle_{\mathbb{C}^d} \le \left \langle f(x) +W(x)f(x), f(x) \right \rangle_{\mathbb{C}^d}, \qquad \forall n \in \mathbb{N}.\]
 Since the right-hand function is integrable, the Dominated Convergence Theorem implies that
\[  \lim_{n \rightarrow \infty} \| f \|^2_{L^2(W_n)} = \int_{\mathbb{R}}  \lim_{n\rightarrow \infty} \left \langle W_n(x) f(x), f(x) \right \rangle_{\mathbb{C}^d} dx = \| f\|_{L^2(W)}^2.\]
To obtain \eqref{eqn:lim2}, first observe that
\[ 
\begin{aligned} 
\sum_{I \in \mathcal{D}}  \left \langle \langle W_n \rangle_I \widehat{f}(I),  \widehat{f}(I) \right \rangle_{\mathbb{C}^d}
&= \frac{1}{n} \sum_{I \in \mathcal{D}}  \left \langle \langle P_{E^n_1(x)} \rangle_I \widehat{f}(I),\widehat{f}(I) \right \rangle_{\mathbb{C}^d} \\
 &+ \sum_{I \in \mathcal{D}}   \left \langle \langle P_{E^n_2(x)} W(x) P_{E^n_2(x)} + nP_{E^n_3(x)} \rangle_I \widehat{f}(I), \widehat{f}(I) \right \rangle_{\mathbb{C}^d}.
 \end{aligned}
 \]
The first term clearly goes to zero as $n \rightarrow \infty.$ Meanwhile, the terms in the second sum are increasing in $n$. So, we can apply the Monotone Convergence Theorem twice to conclude
\[
\begin{aligned}
\lim_{n\rightarrow \infty} \sum_{I \in \mathcal{D}}  \left \langle \langle W_n \rangle_I \widehat{f}(I), \widehat{f}(I) \right \rangle_{\mathbb{C}^d} &= \sum_{I \in \mathcal{D}}  \lim_{n\rightarrow \infty}  \left \langle \langle P_{E^n_2(x)} W(x) P_{E^n_2(x)} + nP_{E^n_3(x)} \rangle_I \widehat{f}(I), \widehat{f}(I) \right \rangle_{\mathbb{C}^d} \\
& =  \sum_{I \in \mathcal{D}}   \left \langle \langle  \lim_{n\rightarrow \infty} P_{E^n_2(x)} W(x) P_{E^n_2(x)} + nP_{E^n_3(x)} \rangle_I \widehat{f}(I), \widehat{f}(I) \right \rangle_{\mathbb{C}^d} \\
& =  \sum_{I \in \mathcal{D}}  \left \langle \langle W \rangle_I \widehat{f}(I), \widehat{f}(I) \right \rangle_{\mathbb{C}^d}. 
\end{aligned}
\]
Now, letting $n \rightarrow \infty$ in \eqref{eqn:bdd3.3} and using  \eqref{eqn:A2n}, \eqref{eqn:lim1}, and\eqref{eqn:lim2} gives Theorem \ref{thm:SquareEst2} for general $W$, since $L^2(W) \cap L^2$ is dense in $L^2(W).$  Theorem \ref{thm:SquareEst1} follows similarly.
\end{remark}

\section{The Hilbert Transform} \label{sec:ht}

The bounds given in Theorems \ref{thm:SquareEst1} and \ref{thm:SquareEst2} imply similar bounds for the Hilbert transform on $L^2(W)$.
To see this, fix $\alpha \in \mathbb{R}$ and $r>0.$ The densely-defined shift operator $\Sha^{\alpha, r}$ on $L^2(\mathbb{R}, \mathbb{C})$  is given by
\[ \Sha^{\alpha, r} f \equiv \frac{1}{\sqrt{2}} \sum_{I \in \mathcal{D}^{\alpha,r}}\widehat{f}(I) \left( h_{I_-} - h_{I_+} \right). \]
In \cite{pet00}, Petermichl showed that
the Hilbert transform $H$ on $L^2(\mathbb{R}, \mathbb{C})$ is basically an average of these dyadic shifts. Specifically, there is
a constant $c$ and $L^{\infty}(\mathbb{R}, \mathbb{C})$ function $b$ such that $H = cT + M_b,$ where $T$ is in the weak operator 
closure of the convex hull of the set $\{ \Sha^{\alpha,r}\}_{\alpha, r}$  in  $\mathcal{L}(L^2(\mathbb{R}, \mathbb{C}))$ and $M_b$ is multiplication by $b$. 
The Hilbert transform on $L^2(\mathbb{R}, \mathbb{C}^d)$, also denoted $H$, is the scalar Hilbert transform applied component-wise. The dyadic 
shift operators $\Sha^{\alpha, r}$ on $L^2(\mathbb{R}, \mathbb{C}^d)$ are similarly defined by
\[\Sha^{\alpha, r} f \equiv \frac{1}{\sqrt{2}} \sum_{I \in \mathcal{D}^{\alpha,r}}\widehat{f}(I) \left( h_{I_-} - h_{I_+} \right), \]
which is the same as applying the scalar  $\Sha^{\alpha, r}$ shifts component-wise.
Using the scalar-result, the Hilbert transform $H$ on $L^2(\mathbb{R}, \mathbb{C}^d)$ satisfies
$H = c \widetilde{T} + M_b$ where $\widetilde{T}$ is $T$ applied component-wise and so, is in the weak operator 
closure of the convex hull of the set $\{\Sha^{\alpha,r}\}_{\alpha, r}$  in  $\mathcal{L}(L^2(\mathbb{R}, \mathbb{C}^d))$.

In \cite{vt97}, Treil and Volberg showed
that for $A_2$ weights $W$,  the Hilbert transform is bounded on $L^2(W)$, 
 but they did not track the dependence on the $A_2$ characteristic $[W]_{A_2}.$
In contrast, using our square function estimates, we are able to establish the following: 

\begin{theorem} \label{thm:ht} Let $W$ be a  $d \times d$  matrix weight in $A_2.$ Then  
\[ \| H  f \|_{L^2(W)} \lesssim [W]_{A_2}^{\frac{3}{2}} \textnormal{log}\, [W]_{A_2}  \| f \|_{L^2(W)} \qquad \forall  f\in L^2(W). \]
\end{theorem}

\begin{proof} As before, we omit the $\alpha, r$ notation. 
Observe that the square function norm in Theorems \ref{thm:SquareEst1} and \ref{thm:SquareEst2}
is not affected by dyadic shifts. Specifically, let $\tilde{I}$ denote the parent of $I$ in the dyadic grid. Then
\[
\begin{aligned}
\| S_W \Sha f \|_{L^2(\mathbb{R}, \mathbb{R})}^2 &= 
\sum_{I \in \mathcal{D}} \left \langle \langle W \rangle_I \widehat{\Sha f}(I),
\widehat{ \Sha f}(I) \right \rangle_{\mathbb{C}^d}  \\
& = 
\frac{1}{2} \sum_{I \in \mathcal{D}} \left \langle \langle W \rangle_I \widehat{f}(\tilde{I}),
 \widehat{f}(\tilde{I}) \right \rangle_{\mathbb{C}^d} \\
 & = \sum_{I \in \mathcal{D}} \left \langle  \frac{1}{2} \left( \langle W \rangle_{I_-} +  \langle W \rangle_{I_+} \right)  \widehat{f}(I),
 \widehat{f}(I) \right \rangle_{\mathbb{C}^d} \\
 & = \sum_{I \in \mathcal{D}} \left \langle  \langle W \rangle_I  \widehat{f}(I),
 \widehat{f}(I) \right \rangle_{\mathbb{C}^d} \\
 & = \|S_W f \|_{L^2(\mathbb{R}, \mathbb{R})}^2.
\end{aligned}
\]
 Now, using Theorems \ref{thm:SquareEst1} and \ref{thm:SquareEst2}, we have
\[
\begin{aligned}
\| \Sha f \|^2_{L^2(W)}  &\lesssim [W]_{A_2} \textnormal{log}\, [W]_{A_2} \| S_W \Sha f \|_{L^2(\mathbb{R}, \mathbb{R})}^2 \\
&=   [W]_{A_2} \textnormal{log}\, [W]_{A_2}  \|S_W f \|_{L^2(\mathbb{R}, \mathbb{R})}^2 \\
 &\lesssim [W]^3_{A_2} ( \textnormal{log}\, [W]_{A_2})^2 \| f \|^2_{L^2(W)}.
\end{aligned}
\]
The formula for $H$ in terms of dyadic shifts implies that
\[ \| H f \|^2_{L^2(W)} \lesssim \sup_{\alpha, r} \| \Sha^{\alpha, r} f \|^2_{L^2(W)}  + \| b \|^2_{\infty} 
\| f \|^2_{L^2(W)} \lesssim [W]^3_{A_2} ( \textnormal{log}\, [W]_{A_2})^2 \| f \|^2_{L^2(W)},
 \]
as desired.
\end{proof}

\section{Haar Multipliers}\label{sec:hm}

The arguments above extend easily to Haar multipliers. To begin, let $\sigma=\{\sigma_I\}_{I\in\mathcal{D}}$ be a sequence of $d \times d$ matrices and recall the Haar multiplier $T_{\sigma}$ defined 
by
$$
T_{\sigma}f \equiv \sum_{I\in\mathcal{D}} \sigma_I \widehat{f}(I) h_I.
$$
To obtain boundedness on $L^2(W)$, it is crucial that the matrices $\sigma_I$ interact well with $W$. To be precise, fix a weight $W \in A_2$ and define
$$\left\Vert \sigma \right\Vert_{\infty} \equiv\inf\left\{C: \left\langle W\right\rangle_I^{-\frac{1}{2}}\sigma_I^{*}\left\langle W\right\rangle_I\sigma_I\left\langle W\right\rangle_I^{-\frac{1}{2}}\leq C^2I_{d\times d} \quad\forall I\in\mathcal{D}\right\}.
$$
Equivalently, we could define $\left\Vert \sigma\right\Vert_{\infty}=\sup_{I\in\mathcal{D}} \big\Vert \left\langle W\right\rangle_I^{\frac{1}{2}}\sigma_I \left\langle W\right\rangle_I^{-\frac{1}{2}}\big\Vert$.  Then, a variant of the following result is established by Isralowitz-Kwon-Pott in \cite{IKP}: 
\begin{theorem}
Let $W\in A_2$ and $\sigma=\{\sigma_I\}_{I\in\mathcal{D}}$ a sequence of matrices.  Then the Haar multiplier $T_{\sigma}$
is bounded on $L^2(W)$ if and only if $\| \sigma \|_{\infty} <\infty$.
\end{theorem}
Here, we have translated their result to the notation of this paper. Now, we provide a new and simpler proof of this boundedness result for $p=2.$ Using our previous arguments, we are also able to track the dependence on $[W]_{A_2}.$

\begin{theorem} \label{thm:HaarMult}
Let $W$ be a  $d \times d$ matrix weight in $A_2$ and let $\sigma=\{\sigma_I\}_{I\in\mathcal{D}}$ be a sequence of $d \times d$ matrices. Then $T_{\sigma}$ is bounded on $L^2(W)$ if and only if $\| \sigma \|_{\infty} < \infty$. Moreover,
$$
\left\Vert T_{\sigma} f\right\Vert_{L^2(W)} \lesssim [W]^{\frac{3}{2}}_{A_2}  \textnormal{log}\, [W]_{A_2} \left\Vert \sigma\right\Vert_{\infty} \| f \|_{L^2(W)} .
$$
\end{theorem}

\begin{proof} Necessity is almost immediate. Fix $I\in\mathcal{D}$ and $e\in\mathbb{C}^d$ and simply set $f \equiv  \langle W  \rangle_I^{-\frac{1}{2}} h_I e.$  Then simple computations prove that $T_{\sigma} f =  \sigma_I  \langle W \rangle_I^{-\frac{1}{2}} h_Ie$ and the following norm equalities:
\[
\begin{aligned}
 \| f\|_{L^2(W)}^2=\| \langle W  \rangle_I^{-\frac{1}{2}} h_Ie \|_{L^2(W)}^2 & =\Vert e\Vert_{\mathbb{C}^d}^2;\\
 \left\Vert T_{\sigma}f\right\Vert_{L^2(W)}^2=\left\Vert \sigma_I  \langle W \rangle_I^{-\frac{1}{2}} h_Ie\right\Vert_{L^2(W)}^2 & =\left\langle \left\langle W\right\rangle_I^{-\frac{1}{2}}\sigma_I^{*}\left\langle W\right\rangle_I\sigma_I\left\langle W\right\rangle_I^{-\frac{1}{2}} e,e\right\rangle_{\mathbb{C}^d}.   
\end{aligned}
 \]
 Assuming $T_{\sigma}$ is bounded on $L^2(W),$ we can then conclude:
$$
\left\langle \left\langle W\right\rangle_I^{-\frac{1}{2}}\sigma_I^{*}\left\langle W\right\rangle_I\sigma_I\left\langle W\right\rangle_I^{-\frac{1}{2}}e,e\right\rangle_{\mathbb{C}^d}= \| T_{\sigma} f \|_{L^2(W)}^2 \leq \left\Vert T_{\sigma}\right\Vert_{L^2(W)\to L^2(W)}^2\left\Vert e\right\Vert_{\mathbb{C}^d}^2.
$$
Since $e\in\mathbb{C}^d$ and $I\in\mathcal{D}$ was arbitrary we have that $\| \sigma \|_{\infty} < \infty$.

The proof of sufficiency, with the desired constant, is largely a repetition of computations from earlier in the paper.  
As before, observe that the square function in Theorems \ref{thm:SquareEst1} and \ref{thm:SquareEst2}
interacts well with Haar multipliers. Specifically, 
\[
\begin{aligned}
\| S_W T_{\sigma} f\|_{L^2(\mathbb{R}, \mathbb{R})}^2  &= 
\sum_{I \in \mathcal{D}} \left \langle \langle W \rangle_I \widehat{T_{\sigma} f}(I),
\widehat{ T_{\sigma} f}(I) \right \rangle_{\mathbb{C}^d} \\
 & = \sum_{I \in \mathcal{D}} \left \langle  \langle W \rangle_I  \sigma_I \widehat{f}(I),
 \sigma_I \widehat{f}(I) \right \rangle_{\mathbb{C}^d}\\
& = \sum_{I \in \mathcal{D}} \left \langle  \langle W \rangle_I  \sigma_I \langle W \rangle_I^{-\frac{1}{2}} \langle W \rangle_I^{\frac{1}{2}}\widehat{f}(I),
 \sigma_I  \langle W \rangle_I^{-\frac{1}{2}} \langle W \rangle_I^{\frac{1}{2}}\widehat{f}(I) \right \rangle_{\mathbb{C}^d}\\
 & \leq  \left\Vert\sigma\right\Vert_{\infty}^2 \sum_{I \in \mathcal{D}} \left \langle   \langle W \rangle_I \widehat{f}(I), \widehat{f}(I) \right \rangle_{\mathbb{C}^d} \\
 & =   \left\Vert\sigma\right\Vert_{\infty}^2 \| S_W f \|_{L^2(\mathbb{R}, \mathbb{R})}^2.
\end{aligned}
\]
Simple applications of Theorems \ref{thm:SquareEst1} and \ref{thm:SquareEst2} then yield
\[
\begin{aligned}
\| T_{\sigma} f \|^2_{L^2(W)}  &\lesssim [W]_{A_2} \textnormal{log}\, [W]_{A_2}  \| S_W T_{\sigma} f\|_{L^2(\mathbb{R}, \mathbb{R})}^2
 \\
& \le   [W]_{A_2} \textnormal{log}\, [W]_{A_2} \left\Vert\sigma\right\Vert_{\infty}^2 \| S_W f \|_{L^2(\mathbb{R}, \mathbb{R})}^2 \\ 
 &\lesssim [W]^3_{A_2} ( \textnormal{log}\, [W]_{A_2})^2 \left\Vert \sigma\right\Vert_{\infty}^2 \| f \|^2_{L^2(W)},
\end{aligned}
\]
which gives the desired bound.
\end{proof}

\begin{remark}  \normalfont One should observe that the arguments in Theorems \ref{thm:ht} and \ref{thm:HaarMult} rest on a good relationship between the operator in question and the square function $S_W.$  Specifically, in Theorem  \ref{thm:ht}, the family of dyadic shifts $\Sha^{\alpha, r}$ interacts well with $S_W$ and by taking averages of them, one can recover the Hilbert transform. It is not hard to show that $S_W$ also interacts well with other similarly-nice dyadic shifts, which one could use to build other Calder\'on-Zygmund operators and obtain similar estimates.  
\end{remark}

\section{Open Questions} \label{sec:open}

\subsection{Square Function Estimates}

If $w$ is a scalar-valued $A_2$ weight, then Theorem \ref{thm:SquareEst1} is true with $[w]^2_{A_2}$ replacing $[w]^2_{A_2} \log [w]_{A_2}$.
This motivates the following conjecture:

\begin{conjecture} \label{con:SquareEst1} Let $W$ be a  $d \times d$ 
matrix weight in $A_2.$ Then 
\[
\| S_W f \|_{L^2(\mathbb{R}, \mathbb{R})}^2 \lesssim 
[ W  ]_{A_2}^2 \| f \|^2_{L^2(W)} \quad \forall f \in L^2(W). 
\]
\end{conjecture}

To prove Conjecture \ref{con:SquareEst1}, we would need to control the term $S_3$ from \eqref{eqn:S31} in a more optimal way. 
Our current method of using Theorem \ref{thm:tv} introduces the troublesome $\textnormal{log}\, [W]_{A_2}$ term. An alternate method of controlling
$S_3$ would use a matrix version of the weighted Carleson Embedding Theorem. One would first 
control $S_3$ by
\[
\begin{aligned}
S_3 & =   \sum_{I \in \mathcal{D}} \sum_{k =1}^d \left | \left \langle  \langle W \rangle_I^{-\frac{1}{2}} \widehat{W}(I)  \langle  W \rangle_I^{-1} \langle f \rangle_I,  e^k_I \right \rangle_{\mathbb{C}^d} \right|^2 \\
& \lesssim \sum_{I \in \mathcal{D}}  \left \|  \langle W \rangle_I^{-\frac{1}{2}} \widehat{W}(I)  \langle  W \rangle_I^{-1} \langle f \rangle_I \right \|^2_{\mathbb{C}^d} \\
& =  \sum_{I \in \mathcal{D}}\left \langle  \langle  W \rangle_I^{-1} \widehat{W}(I)    \langle W \rangle_I^{-1} \widehat{W}(I)  \langle  W \rangle_I^{-1} \langle f \rangle_I ,  \langle f \rangle_I  \right \rangle_{\mathbb{C}^d}. \\
\end{aligned}
\]
Conjecture \ref{con:SquareEst1} would follow if we could show
\[ \sum_{I \in \mathcal{D}}\left \langle  \langle  W \rangle_I^{-1} \widehat{W}(I)    \langle W \rangle_I^{-1} \widehat{W}(I)  \langle  W \rangle_I^{-1} \langle f \rangle_I ,  \langle f \rangle_I  \right \rangle_{\mathbb{C}^d}\lesssim [W]_{A_2} \| f\|^2_{L^2(W^{-1})}.\]
To obtain this, we need two things. First, we need a matrix version of the weighted Carleson Embedding Theorem. Very recently, the needed result has actually been proven by Culiuc-Treil in \cite{ct15}, who show the following:

\begin{theorem} \label{thm:CETnew} Let $W$ be a $d \times d$ matrix weight and let  $\{ A_I\}_{I \in \mathcal{D}}$ be a sequence of positive semidefinite $d\times d$ matrices indexed by the dyadic intervals. Then
\[
\sum_{I\in\mathcal{D}} \left\langle A_I \left\langle f\right\rangle_I, \left\langle f\right\rangle_I\right\rangle_{\mathbb{C}^d} \lesssim C\left\Vert f\right\Vert_{L^2(W^{-1})}^2  \qquad \forall f \in L^2(W^{-1})
\]
if and only if 
\[
\frac{1}{\left\vert J\right\vert}\sum_{I\subset J} \left\langle W\right\rangle_I A_I \left\langle W\right\rangle_I\leq C \left\langle W\right\rangle_J \qquad \forall  J \in \mathcal{D}.
\]
\end{theorem}

Second, we need the appropriate testing conditions to apply Theorem \ref{thm:CETnew}. Specifically The $A_I$ that appear in our bound for $S_3$ are the matrices
\[ \langle W \rangle_I^{-1} \widehat{W}(I)    \langle W \rangle_I^{-1} \widehat{W}(I)  \langle  W \rangle_I^{-1}.\]
Given Theorem \ref{thm:CETnew}, we need the appropriate testing condition to apply it to $S_3.$
Indeed, we require
\begin{equation} \label{eqn:testingnew}
\frac{1}{\left\vert J\right\vert}\sum_{I\subset J}  \widehat{W}(I)    \langle W \rangle_I^{-1} 
\widehat{W}(I)  \lesssim [W]^2_{A_2} \left \langle W\right\rangle_{J}, \quad \forall J \in \mathcal{D}.
\end{equation}
In the scalar case, this is proved by Wittwer in \cite{wit00} using estimates from Buckley \cite{buc93}. Hukovic-Treil-Volberg give a Bellman function proof in \cite{htv00}. Neither of these arguments adapt well to the matrix setting and currently, it is not clear whether \eqref{eqn:testingnew} is true for matrices.

\subsection{The Hilbert Transform and Haar Multipliers}

If $w$ is a scalar $A_2$ weight, then Theorems \ref{thm:ht} and \ref{thm:HaarMult}  
are true with $[w]_{A_2}$ replacing $[w]^{\frac{3}{2}}_{A_2} \log [w]_{A_2}$.
This motivates the following conjecture:

\begin{conjecture} \label{con:ht} Let $W$ be a  $d \times d$  matrix weight in $A_2$ and let $\{ \sigma_I\}_{I \in \mathcal{D}}$ be a sequence
of matrices satisfying $\| \sigma \|_{\infty} < \infty.$ Then  
\[ 
\begin{aligned}
&\| H   \|_{L^2(W) \rightarrow L^2(W)} \lesssim [W]_{A_2}; \\
 &\| T_{\sigma}   \|_{L^2(W) \rightarrow L^2(W)} \lesssim [W]_{A_2} \left\Vert \sigma\right\Vert_{\infty}.
\end{aligned}
\]
\end{conjecture}

Given our current tools, those estimates seem out of reach. However, if we could prove
Conjecture \ref{con:SquareEst1} by establishing a bound of $[W]_{A_2}$ in Theorem \ref{thm:SquareEst2},
then the arguments from the proofs of Theorems \ref{thm:ht} and \ref{thm:HaarMult} would immediately imply that 
\[ 
\begin{aligned}
&\| H   \|_{L^2(W) \rightarrow L^2(W)} \lesssim   [W]_{A_2}^{\frac{3}{2}}; \\
&\| T_{\sigma}   \|_{L^2(W) \rightarrow L^2(W)} \lesssim  [W]_{A_2}^{\frac{3}{2}} \left\Vert \sigma\right\Vert_{\infty}.
\end{aligned}
 \]

\subsection{Extrapolation} In the scalar case, after Hyt\"onen established
\[ \| T \|_{L^2(w) \rightarrow L^2(w)} \lesssim [w]_{A_2},\]
for all Calder\'on-Zgymund operators $T$, he used the sharp form of Rubio de Francia's extrapolation theorem due to Dragi\u{c}evi\'c-Grafakos-Pereyra-Petermichl \cite{dgpp05} to obtain the following 
\[ \| T\|_{L^p(w) \rightarrow L^p(w)} \lesssim [w]_{A_p}^{ \max \{1, \frac{1}{p-1} \}} \qquad 1< p < \infty,\]
which is sharp for all exponents. For the dyadic square function, one can also extrapolate weighted $L^p$ bounds from the weighted $L^2$ bounds, but the estimates are only sharp for $1< p \le 2.$ These extrapolation results rely heavily on maximal functions; Rubio de Francia's original extrapolation theorem \cite{ru84} used the connections between the maximal function and $A_p$ weights. Similarly, the sharp theorem in \cite{dgpp05} used the sharp dependence of the maximal function's norm on $[w]_{A_p}$:
\[ \| M \|_{L^p(w) \rightarrow L^p(w)} \lesssim [w]_{A_p}^{\frac{p'}{p}} \qquad 1< p <\infty,\]
first proved by Buckley in \cite{buc93b}. Then, a natural question is: 
\begin{center} Can one use extrapolation and the operator bounds from Theorems \ref{thm:SquareEst1}, \ref{thm:ht}, and \ref{thm:HaarMult} to deduce similar bounds for operators related to $A_p$ weights, $1<p<\infty?$
\end{center}
This is an interesting question certainly worth exploring. However, there are several complications stemming from maximal functions in the vector-valued setting, which make the question difficult. First, a problem arises when defining maximal operators in the vector case. For example, one could define $Mf(x)$ to be the average of $f$ over an interval containing $x$ with largest vector magnitude. However, this ignores the fact that the effect of a matrix weight $W$ will depend both on direction and magnitude. Instead,  in \cite{cg01, gol03}, Christ-Goldberg and Goldberg studied weighted, vector analogues of the maximal function and defined a different operator $M^p_W$ for each weight $W$ and $1<p<\infty$. For the exact formulas, see  \eqref{eqn:max} in Section \ref{sec:square}.
The boundedness of these maximal operators is closely related to the weight $W$ belonging to a specific $A_p$ class. However, proving related extrapolation results seems difficult because we now have a  family of maximal operators that rely on both the weight $W$ and the exponent $p$. Furthermore, although Isralowitz-Kwon-Pott in \cite{IKP} established the nice sharp bound
\[ \| M^2_W  \|_{L^2(\mathbb{R}, \mathbb{C}^d) \rightarrow L^2(\mathbb{R}, \mathbb{R})} \lesssim [W]_{A_2},\]
the sharp bounds for $p \ne 2$ are currently unknown. It is worth noting that Isralowitz-Kwon-Pott indicate in \cite{IKP}  that similar bounds for $p\ne 2$ will be established in \cite{im15}.


\begin{thebibliography}{12}

\bibitem{bw15b}
K. Bickel and B.D. Wick. 
Well-localized operators on matrix weighted $L^2$ spaces. 
To appear in \emph{Houston Math. J.}, available at \href{http://arxiv.org/abs/1407.3819}{http://arxiv.org/abs/1407.3819}.

\bibitem{bw15}
K. Bickel and B.D. Wick.
A study of the matrix {Carleson} embedding theorem with applications to sparse operators.
Preprint , available at \href{http://arxiv.org/abs/1503.06493}{http://arxiv.org/abs/1503.06493}.


\bibitem{buc93}
S. Buckley.
Summation conditions on weights. 
\emph{Michigan Math. J.} \textbf{40} (1993), no. 1, 153--170.

\bibitem{buc93b}
S. Buckley.
Estimates for operator norms on weighted spaces and reverse {J}ensen inequalities. 
\emph{Trans. Amer. Math. Soc.} \textbf{340} (1993), no. 1, 253--272. 

\bibitem{cg01}
M. Christ and M. Goldberg.
Vector $A_2$ weights and a {Hardy-Littlewood} maximal function.
\emph{Trans. Amer. Math. Soc.} \textbf{353} (2001), no. 5, 1995--2002.

\bibitem{ct15}
A. Culiuc and S. Treil. 
The {Carleson} embedding theorem with matrix weights.
Preprint, available at \href{http://arxiv.org/abs/1508.01716}{http://arxiv.org/abs/1508.01716}.


\bibitem{dgpp05}
O. Dragi\u{c}evi\'c, L. Grafakos, M.C.~Pereyra, and S. Petermichl.
Extrapolation and sharp norm estimates for classical operators on weighted {L}ebesgue spaces.
\emph{Publ. Mat.} \textbf{49} (2005), no. 1, 73--91. 

\bibitem{gptv01}
T.A. Gillespie, S. Pott, S. Treil, and A. Volberg.
Logarithmic growth for matrix Haar multipliers.
\emph{J. London Math. Soc. (2)} \textbf{64} (2001), no. 3, 624--636. 

\bibitem{gptv04}
T.A. Gillespie, S. Pott, S. Treil, and A. Volberg.
Logarithmic growth for weighted {Hilbert transforms and vector Hankel operators.} 
\emph{J. Operator Theory} \textbf{52} (2004), no. 1, 103--112. 


\bibitem{gol03}
M. Goldberg.
Matrix $A_p$ weights via maximal functions. 
\emph{Pacific J. Math.} \textbf{211} (2003), no. 2, 201--220. 

\bibitem{gol02} 
M. Goldberg.
Asymptotic properties of the vector {C}arleson embedding theorem. 
\emph{Proc. Amer. Math. Soc.} \textbf{130} (2002), no. 2, 529--531. 

\bibitem{th12}
 T.P. Hyt{\"o}nen. The sharp weighted bound for general {Calder\'on-Zygmund} operators.
 \emph{Ann. of Math. (2).} \textbf{175} (2012), no. 3, 1473--1506.

\bibitem{htv00}
S. Hukovic, S. Treil, and A. Volberg.
The Bellman functions and sharp weighted inequalities for square functions. 
\emph{Complex analysis, operators, and related topics}, 97--113, 
\emph{Oper. Theory Adv. Appl.}, \textbf{113}, Birkhäuser, Basel, 2000. 

\bibitem{k97}
N.H. Katz.
Matrix valued paraproducts. 
\emph{Proceedings of the conference dedicated to Professor Miguel de Guzm\'an} (El Escorial, 1996). 
\emph{J. Fourier Anal. Appl.} \textbf{3} (1997), Special Issue, 913--921. 

\bibitem{kp97}
N.H. Katz and C. Pereyra.
On the two weights problem for the {H}ilbert transform. 
\emph{Rev. Mat. Iberoamericana} \textbf{13} (1997), no. 1, 211--243. 


\bibitem{k1}
R. Kerr.
{T}oeplitz products and two-weight inequalities on
spaces of vector-valued functions.
Thesis (Ph.D.)-University of Glasgow. 2011.

\bibitem{k2}
R. Kerr.
Martingale transforms, the dyadic shift and the {H}ilbert transform: a sufficient condition for boundedness between matrix weighted spaces.
Preprint, available at \href{http://arxiv.org/abs/0906.4028}{http://arxiv.org/abs/0906.4028}.

\bibitem{IKP}
J. Isralowitz, H. Kwon, and S. Pott.
Matrix weighted norm inequalities for commutators and paraproducts with matrix symbols.
Preprint, available at \href{http://arxiv.org/abs/1507.04032}{http://arxiv.org/abs/1507.04032}.

\bibitem{I15}
J. Isralowitz. 
Matrix weighted Triebel-Lizorkin bounds: a simple stopping-time proof. 
Preprint, available at t \href{http://arxiv.org/abs/1507.06700}{http://arxiv.org/abs/1507.06700}.

\bibitem{im15}
J. Isralowitz and K. Moen. 
Matrix weighted {P}oincare inequalities and applications
to degenerate elliptic systems, Preprint.

\bibitem{lpr10}
M. Lacey, S. Petermichl, and M.C. Reguera.
Sharp $A_2$ inequality for {H}aar shift operators.
\emph{Math. Ann.} \textbf{348} (2010), no. 1, 127--141.

\bibitem{lt07}
M. Lauzon and S. Treil.
Scalar and vector {Muckenhoupt} weights.
\emph{Indiana Univ. Math. J.} \textbf{56} (2007), no. 4, 1989--2015. 

\bibitem{ntv97}
F. Nazarov, S. Treil, and A. Volberg.
Counterexample to the infinite-dimensional {C}arleson embedding theorem. 
\emph{C. R. Acad. Sci. Paris Sér. I Math.} \textbf{325} (1997), no. 4, 383--388. 

\bibitem{nt97}
F. Nazarov and S. Treil.
The hunt for a {B}ellman function: applications to estimates for singular integral operators and to other classical problems of harmonic analysis. (Russian) 
\emph{Algebra i Analiz.} \textbf{8} (1996), no. 5, 32--162; translation in 
\emph{St. Petersburg Math. J.} \textbf{8} (1997), no. 5, 721--824. 


\bibitem{nptv02}
F. Nazarov, G. Pisier, S. Treil, and A. Volberg.
Sharp estimates in vector Carleson imbedding theorem and for vector paraproducts. 
\emph{J. Reine Angew. Math.} \textbf{542} (2002), 147--171. 

\bibitem{nr15}
M. Nielsen and M.G.~Rasmussen.
Projection operators on matrix weighted $L^p$ and a simple sufficient Muckenhoupt condition.
Preprint, available at \href{http://arxiv.org/pdf/1503.01961.pdf}{http://arxiv.org/pdf/1503.01961.pdf}.


\bibitem{pet00}
S. Petermichl.
Dyadic shifts and a logarithmic estimate for Hankel operators with matrix symbol.
\emph{C. R. Acad. Sci. Paris S\'er. I Math.} \textbf{330} (2000), no. 6, 455--460. 

\bibitem{pet07}
S. Petermichl
The sharp bound for the Hilbert transform on weighted {L}ebesgue 
spaces in terms of the classical $A_p$ characteristic.
\emph{Amer. J. Math.} \textbf{129} (2007), no. 5, 1355--1375. 

\bibitem{p08}
S. Petermichl.
The sharp weighted bound for the Riesz transforms.
\emph{Proc. Amer. Math. Soc.} \textbf{136} (2008), no. 4, 1237--1249. 

\bibitem{petpot02}
S. Petermichl and S. Pott.
An estimate for weighted {H}ilbert transform via square functions.
\emph{Trans. Amer. Math. Soc.} \textbf{354} (2002), no. 4, 1699--1703.

\bibitem{pp03}
S. Petermichl and S. Pott.
A version of {B}urkholder's theorem for operator-weighted spaces. 
\emph{Proc. Amer. Math. Soc}. \textbf{131} (2003), no. 11, 3457--3461.

\bibitem{pv02}
S. Petermichl and A. Volberg.
Heating of the {Ahlfors-Beurling} operator: weakly quasiregular maps on the plane are quasiregular. 
\emph{Duke Math. J.} \textbf{112} (2002), no. 2, 281--305. 

\bibitem{p07}
S. Pott.
A sufficient condition for the boundedness of operator-weighted Haar multipliers and {H}ilbert transform. 
\emph{Studia Math.} \textbf{182} (2007), no. 2, 99--111.

\bibitem{ps12}
S. Pott and A. Stoica.
Linear bounds for Calder\'on-Zygmund operators with even kernel on UMD spaces.
\emph{J. Funct. Anal} \textbf{266} (2014), no. 5, 3303--3319.

\bibitem{ps15}
S. Pott and A. Stoica.
Bounds for Calder\'on-Zygmund operators with matrix $A_2$ weights.
Preprint, available at \href{http://arxiv.org/abs/1508.06408}{http://arxiv.org/abs/1508.06408}.

\bibitem{rvv10}
A. Reznikov, V. Vasyunin, V. Volberg.
An observation: cut-off of the weight $w$ does not increase the $A_{p1,p2}$-``norm'' of $w$.
Preprint, available at \href{http://arxiv.org/abs/1008.3635}{http://arxiv.org/abs/1008.3635}.

\bibitem{s03}
S. Roudenko.
Matrix-weighted Besov spaces. 
\emph{Trans. Amer. Math. Soc.} \textbf{355} (2003), no. 1, 273--314. 

\bibitem{ru84}
J. Rubio de Francia. 
Factorization theory and $A_p$ weights. 
\emph{Amer. J. Math.} \textbf{106} (1984), no. 3, 533--547. 

\bibitem{t12}
S. Treil.
Sharp $A_2$ estimates of Haar shifts via Bellman function. Preprint, available at \href{http://arxiv.org/abs/1105.2252}{http://arxiv.org/abs/1105.2252}.

\bibitem{vt97b}
S. Treil and A. Volberg.
Continuous frame decomposition and a vector {Hunt-Muckenhoupt-Wheeden} theorem. 
\emph{Ark. Mat.} \textbf{35} (1997), no. 2, 363--386. 

\bibitem{vt97} 
S. Treil and A. Volberg.
Wavelets and the angle between past and future. 
\emph{J. Funct. Anal.} \textbf{143} (1997), no. 2, 269--308. 

\bibitem{vol97}
A. Volberg.
Matrix $A_p$ weights via {S}-functions. 
\emph{J. Amer. Math. Soc.} \textbf{10} (1997), no. 2, 445--466. 


\bibitem{wit00}
J. Wittwer.
A sharp estimate on the norm of the Haar multiplier. 
\emph{Math. Res. Lett.} \textbf{7} (2000), no. 1, 1--12. 

\end{thebibliography}
\end{document}